\documentclass[12pt, reqno]{amsart}
\usepackage{amsmath, amsthm, amscd, amsfonts, amssymb, graphicx, color}
\usepackage[bookmarksnumbered, colorlinks, plainpages]{hyperref}
\input{mathrsfs.sty}

\newtheorem{theorem}{Theorem}[section]

\newtheorem{proposition}[theorem]{Proposition}
\newtheorem{corollary}[theorem]{Corollary}
\theoremstyle{definition}

\newtheorem{conjecture}[theorem]{Conjecture}

\theoremstyle{remark}
\newtheorem{remark}[theorem]{Remark}
\numberwithin{equation}{section}

\begin{document}
\title{Operator inequalities related to the Corach--Porta--Recht inequality}
\author[C. Conde, M.S. Moslehian, A. Seddik]{Cristian Conde$^1$, Mohammad Sal Moslehian$^2$ and Ameur Seddik$^3$}
\address{$^1$ Instituto de Ciencias, Universidad Nacional de Gral. Sarmiento, J.
M. Gutierrez 1150, (B1613GSX) Los Polvorines and
Instituto Argentino de Matemática ``Alberto P. Calder\'on", Saavedra 15 3º piso,
(C1083ACA) Buenos Aires\\
Argentina}
\email{cconde@ungs.edu.com}
\address{$^2$ Department of Pure Mathematics, Center of Excellence in Analysis on Algebraic Structures (CEAAS), Ferdowsi University of
Mashhad, P.O. Box 1159, Mashhad 91775, Iran}
\email{moslehian@ferdowsi.um.ac.ir and moslehian@member.ams.org}
\address{$^3$ Department of Mathematics, Faculty of Science, Tabuk University, Saudi Arabia}
\email{seddikameur@hotmail.com}


\begin{abstract}
We prove some refinements of an inequality due to X. Zhan in an arbitrary complex Hilbert space by using some results on the Heinz inequality. We present several related inequalities as well as new variants of the Corach--Porta--Recht inequality. We also characterize the class of operators satisfying $\left\Vert SXS^{-1}+S^{-1}XS+kX\right\Vert
\geq (k+2)\left\Vert X\right\Vert$ under certain conditions.
\end{abstract}

\maketitle

\section{Introduction}

Let $\mathbb{B}(\mathscr{H})$, $\mathfrak{I}(\mathscr{H})$ and $\mathfrak{U}(\mathscr{H})$ be the
$C^*$-algebra of all bounded linear operators acting on a complex Hilbert space
$\mathscr{H}$, the set of all invertible elements in $\mathbb{B}(\mathscr{H})$ and the class
of all unitary operators in $\mathbb{B}(\mathscr{H})$, respectively. The operator norm on $\mathbb{B}(\mathscr{H})$ is denoted by $\Vert\cdot\Vert$. We denote by

\begin{enumerate}
\item[$\bullet $] $\mathscr{S}_{0}(\mathscr{H}),$ the set of all invertible
self-adjoint operators in $\mathbb{B}(\mathscr{H}),$

\item[$\bullet $] $\mathcal{P}(\mathscr{H}),$ the set of all positive operators in
$\mathbb{B}(\mathscr{H}),$

\item[$\bullet $] $\mathcal{P}_{0}(\mathscr{H}),$ the set of all invertible positive
operators in $\mathbb{B}(\mathscr{H}),$

\item[$\bullet $] $\mathfrak{U}_{r}(\mathscr{H})=\mathscr{S}_{0}(\mathscr{H})\cap \mathfrak{U}
(\mathscr{H}),$ the set of all unitary reflection operators in $\mathbb{B}(\mathscr{H}),$

\item[$\bullet $] $\mathfrak{N}_{0}(\mathscr{H})$, the set of all invertible normal
operators in $\mathbb{B}(\mathscr{H}).$
\end{enumerate}
For $1\leq p<\infty$, the Schatten $p$-norm class consists of all compact operators $A$ for which $\|A\|_{p}:=({\rm tr}|A|^{p})^{1/p}<\infty$,
where ${\rm tr}$ is the usual trace functional. If $A$ and $B$ are operators in ${\mathbb
B}({\mathscr H})$ we use $A\oplus B$ to denote the $2\times2$
operator matrix $\left[\begin{array}{cc}
 A & 0 \\
  0 & B
\end{array}\right]$, regarded as an operator on ${\mathscr H}\oplus {\mathscr H}$. One can show that
\begin{eqnarray}\label{erf}
\|A\oplus B\|=\max(\|A\|,\|B\|),\quad \|A\oplus
B\|_{p}=\left(\|A\|_{p}^{p}+\|B\|_{p}^{p}\right)^{1/p}\end{eqnarray}
One of the most essential inequalities in the operator theory is the following so-called Heinz inequality:
\begin{eqnarray}\label{1}
\left\Vert
PX+XQ\right\Vert \geq \left\Vert P^{\alpha }XQ^{1-\alpha }+P^{1-\alpha
}XQ^{\alpha }\right\Vert
\end{eqnarray}
for all $P,Q\in \mathcal{P}(\mathscr{H})$, all $X\in \mathbb{B}(\mathscr{H})$ and all $\alpha \in [0,1]$. The proof given by Heinz \cite{4} is based on the complex analysis and is somewhat
complicated. In \cite{5}, McIntosh showed that the Heinz inequality is a consequence of
the following inequality
\begin{eqnarray}\label{2}
\forall A,B,X\in \mathbb{B}(\mathscr{H}),\ \left\Vert A^*AX+XBB^{*
}\right\Vert \geq 2\left\Vert AXB\right\Vert
\end{eqnarray}
McIntosh proved that \eqref{2} holds and gave his ingenious proof of $\eqref{2}\Rightarrow \eqref{1}$. In the literature, inequality \eqref{2} is called ``Arithmetic-geometric-Mean Inequality''.

In \cite{2} Corach--Porta--Recht proved the following, so-called C-P-R inequality,
\begin{eqnarray}\label{3}
\forall S\in \mathscr{S}_{0}(\mathscr{H}) \forall X\in \mathbb{B}(\mathscr{H}),\ \left\Vert
SXS^{-1}+S^{-1}XS\right\Vert \geq 2\left\Vert X\right\Vert
\end{eqnarray}
The C-P-R inequality is a key factor in their study of differential geometry of
self-adjoint operators. They proved this inequality by using the integral
representation of a self-adjoint operator with respect to a spectral measure.

An immediate consequence of the C-P-R inequality is the following:
\begin{eqnarray}\label{4}
\forall S,T\in \mathscr{S}_{0}(\mathscr{H}) \forall X\in \mathbb{B}(\mathscr{H}),\ \left\Vert
SXT^{-1}+S^{-1}XT\right\Vert \geq 2\left\Vert X\right\Vert
\end{eqnarray}
Using the polar decomposition of an operator, we may deduce easily from the
C-P-R inequality the following operator inequality
\begin{eqnarray}\label{5}
\forall S\in \mathfrak{I}(\mathscr{H}) \forall X\in \mathbb{B}(\mathscr{H}),\ \left\Vert S^{*
}XS^{-1}+S^{-1}XS^*\right\Vert \geq 2\left\Vert X\right\Vert
\end{eqnarray}

Three years after and in \cite{3}, Fujii--Fujii--Furuta--Nakamato proved that
inequalities \eqref{1}, \eqref{2}, \eqref{3}, \eqref{4} and two other ones hold and are mutually equivalent.
By giving an easy proof of one of them, they showed a simplified proof of
Heinz inequality, see also \cite{FFN}. Also, it is easy to see that two inequalities \eqref{3} and \eqref{5}
are equivalent.

In \cite{6}, it is shown that the operator inequality
\begin{eqnarray}\label{6}
\forall X\in \mathbb{B}(\mathscr{H}),\ \left\Vert SXS^{-1}+S^{-1}XS\right\Vert \geq
2\left\Vert X\right\Vert ,\ \left( S\in \mathfrak{I}(\mathscr{H})\right)
\end{eqnarray}
is in fact a characterization of $\mathbb{C}^*\mathscr{S}_{0}(\mathscr{H})=\{\lambda M: \lambda \in \mathbb{C}\setminus\{0\}, M \in \mathscr{S}_{0}(\mathscr{H})\}$.

Recently in \cite{7}, using inequality \eqref{5} and the above
characterization of $\mathbb{C}^*\mathscr{S}_{0}(\mathscr{H})$, it is proved
that this class is also characterized by each of the following statements:
\begin{eqnarray}\label{7}
\forall X\in \mathbb{B}(\mathscr{H}),\ \left\Vert SXS^{-1}+S^{-1}XS\right\Vert
=\left\Vert S^*XS^{-1}+S^{-1}XS^*\right\Vert \,\, \left(S\in
\mathfrak{I}(\mathscr{H})\right)
\end{eqnarray}

\begin{eqnarray}\label{8}
\forall X\in \mathbb{B}(\mathscr{H}),\ \left\Vert SXS^{-1}+S^{-1}XS\right\Vert \geq
\left\Vert S^*XS^{-1}+S^{-1}XS^*\right\Vert \,\, \left( S\in
\mathfrak{I}(\mathscr{H})\right)
\end{eqnarray}

Note that this class of operators is the class of all invertible normal
operators in $\mathbb{B}(\mathscr{H})$ the spectrum of which is included in a
straight line passing through the origin.

For the class of all invertible normal operators in $\mathbb{B}(\mathscr{H}),$ it is proved \cite{7, 8} that this class is characterized by each of the
following properties

\begin{eqnarray}\label{9}
\forall X\in \mathbb{B}(\mathscr{H}),\ \left\Vert SXS^{-1}\right\Vert +\left\Vert
S^{-1}XS\right\Vert \geq 2\left\Vert X\right\Vert \,\, \left( S\in \mathfrak{I}
(\mathscr{H})\right)
\end{eqnarray}

\begin{eqnarray}\label{10}
\forall X\in \mathbb{B}(\mathscr{H}),\ \left\Vert SXS^{-1}\right\Vert +\left\Vert
S^{-1}XS\right\Vert =\left\Vert S^*XS^{-1}\right\Vert +\left\Vert
S^{-1}XS^*\right\Vert \,\, \left( S\in \mathfrak{I}(\mathscr{H})\right)
\end{eqnarray}

\begin{eqnarray}\label{11}
\forall X\in \mathbb{B}(\mathscr{H}),\ \left\Vert SXS^{-1}\right\Vert +\left\Vert
S^{-1}XS\right\Vert \geq \left\Vert S^*XS^{-1}\right\Vert +\left\Vert
S^{-1}XS^*\right\Vert \,\, \left( S\in \mathfrak{I}(\mathscr{H})\right)
\end{eqnarray}

\begin{eqnarray}\label{12}
\forall X\in \mathbb{B}(\mathscr{H}),\left\Vert SXS^{-1}\right\Vert +\left\Vert
S^{-1}XS\right\Vert \leq \left\Vert S^*XS^{-1}\right\Vert +\left\Vert
S^{-1}XS^*\right\Vert \,\, \left( S\in \mathfrak{I}(\mathscr{H})\right)
\end{eqnarray}

It is natural to ask what happen if we consider in each of the above
operator inequalities instead of ``$\geq$'', either ``$\leq$'' or ``$=$''.

Let us consider the following associated operator inequalities

\begin{eqnarray}\label{13}
\forall X\in \mathbb{B}(\mathscr{H}),\ \left\Vert SXS^{-1}+S^{-1}XS\right\Vert \leq
2\left\Vert X\right\Vert \,\, \left( S\in \mathfrak{I}(\mathscr{H})\right)
\end{eqnarray}

\begin{eqnarray}\label{14}
\forall X\in \mathbb{B}(\mathscr{H}),\ \left\Vert SXS^{-1}+S^{-1}XS\right\Vert
=2\left\Vert X\right\Vert \,\, \left( S\in \mathfrak{I}(\mathscr{H})\right)
\end{eqnarray}

\begin{eqnarray}\label{15}
\forall X\in \mathbb{B}(\mathscr{H}),\ \left\Vert SXS^{-1}+S^{-1}XS\right\Vert \leq
\left\Vert S^*XS^{-1}+S^{-1}XS^*\right\Vert \,\, \left( S\in
\mathfrak{I}(\mathscr{H})\right)
\end{eqnarray}

\begin{eqnarray}\label{16}
\forall X\in \mathbb{B}(\mathscr{H}),\ \left\Vert SXS^{-1}\right\Vert +\left\Vert
S^{-1}XS\right\Vert =2\left\Vert X\right\Vert \,\, \left( S\in \mathfrak{I}
(\mathscr{H})\right)
\end{eqnarray}

\begin{eqnarray}\label{17}
\forall X\in \mathbb{B}(\mathscr{H}),\ \left\Vert SXS^{-1}\right\Vert +\left\Vert
S^{-1}XS\right\Vert \leq 2\left\Vert X\right\Vert \,\, \left( S\in \mathfrak{I}
(\mathscr{H})\right)
\end{eqnarray}

In \cite{7, 8, 9}, it was established that each of inequalities \eqref{13}, \eqref{16} and \eqref{17} characterize $\mathbb{R}^{* }\mathfrak{U}(\mathscr{H})$ and \eqref{14}
characterizes $\mathbb{C}^*\mathfrak{U}_{r}(\mathscr{H})$.

We found also in \cite{7, 8} that $\mathbb{R}^*\mathfrak{U}(\mathscr{H})$ is also
characterized by each of the following two operator equalities

\begin{eqnarray}\label{18}
\forall X\in \mathbb{B}(\mathscr{H}),\left\Vert S^*XS^{-1}+S^{-1}XS^{*
}\right\Vert =2\left\Vert X\right\Vert \,\, \left( S\in \mathfrak{I}(\mathscr{H})\right)
\end{eqnarray}
\begin{eqnarray}\label{19}
\forall X\in \mathbb{B}(\mathscr{H}),\left\Vert S^*XS^{-1}\right\Vert
+\left\Vert S^{-1}XS^*\right\Vert =2\left\Vert X\right\Vert \,\, \left(
S\in \mathfrak{I}(\mathscr{H})\right)
\end{eqnarray}

A unitarily invariant norm $\left|\left|\left|\cdot\right|\right|\right|$ is defined on a norm ideal $\mathfrak{J}_{\left|\left|\left| .\right|\right|\right|}$ of $\mathbb{B}(\mathscr{H})$ associated with it and has the property
$\left|\left|\left| UXV\right|\right|\right|=\left|\left|\left| X\right|\right|\right|$, where $U$ and $V$ are unitaries and $X \in \mathfrak{J}_{\left|\left|\left| .\right|\right|\right|}$. Note that inequalities \eqref{1}, \eqref{2}, \eqref{3}, \eqref{4} were
generalized for arbitrary unitarily invariant norms. Furthermore, it is proved in \cite{Co} that the characterization of the invertible normal
operators via inequalities of the uniform norm in $\mathbb{B}(\mathscr{H})$
(\eqref{9}-\eqref{13}) also holds for any unitarily invariant norm.

In \cite{12} and in the case $\dim \mathscr{H}<\infty $, by introducing two parameters $r$ and $t$, Zhan proved that for $n \times n$ positive matrices $
A, B$, arbitrary $n \times n$ matrix $X$ and $(t,r)\in (-2,2]\times \lbrack \frac{1}{2},
\frac{3}{2}],$ the following inequality
\begin{eqnarray}\label{21}
(2+t)\left|\left|\left|
A^{r}XB^{2-r}+A^{2-r}XB^{r}\right|\right|\right| \leq 2\left|\left|\left| A^{2}X+tAXB+XB^{2}\right|\right|\right|
\end{eqnarray}
holds for any unitarily invariant norm $\left|\left|\left| .\right|\right|\right|$.
The tool used for proving this inequality is based on the induced Schur product
norm. It should be noted that the case $r=1,\ t=0$ of this result is the
well-known arithmetic-geometric mean inequality due to Bhatia and Davis
\cite{BD}. In this paper we want to extend it and to obtain some refinements of this
inequality to the case where $\mathscr{H}$ is a Hilbert space of arbitrary
dimension by using elementary techniques. We also characterize the class of
operators satisfying $\left\Vert SXS^{-1}+S^{-1}XS+kX\right\Vert
\geq (k+2)\left\Vert X\right\Vert$ under certain conditions.

Recently, Kittaneh proved in \cite{Ki} the following refinement of the
Heinz inequality.

\begin{proposition}
Let $A,B\in \mathcal{P}(\mathscr{H})$ and $X \in \mathfrak{J}_{\left|\left|\left| .\right|\right|\right|}$. Then
\begin{enumerate}
\item for $\alpha\in[0,\frac12]$ the following inequalities hold
\begin{eqnarray}\label{kit1}
\left|\left|\left| A^{\alpha }XB^{1-\alpha }+A^{1-\alpha
}XB^{\alpha }\right|\right|\right|&\leq& \left|\left|\left| A^{\alpha/2 }XB^{1-\alpha/2 }+A^{1-\alpha/2
}XB^{\alpha/2 }\right|\right|\right|\nonumber\\
&\leq&
\frac{1}{\alpha}\int_0^{\alpha}\left|\left|\left| A^{\nu }XB^{1-\nu }+A^{1-\nu
}XB^{\nu }\right|\right|\right|d\nu \nonumber \\
&\leq& \frac 12\left|\left|\left| AX+XB\right|\right|\right|+\frac 12
\left|\left|\left| A^{\alpha }XB^{1-\alpha }+A^{1-\alpha
}XB^{\alpha }\right|\right|\right|\nonumber\\
&\leq& \left|\left|\left| AX+XB\right|\right|\right| \
\end{eqnarray}

\item for $\alpha\in[\frac12,1]$ the following inequalities hold
\begin{eqnarray}\label{kit2}
\left|\left|\left| A^{\alpha }XB^{1-\alpha }+A^{1-\alpha
}XB^{\alpha }\right|\right|\right|&\leq& \left|\left|\left| A^{\frac{1+\alpha}{2}}XB^{\frac{1-\alpha}{2}}+A^{\frac{1-\alpha}{2}}XB^{\frac{1+\alpha}{2}}\right|\right|\right|\nonumber\\
&\leq&
\frac{1}{1-\alpha}\int_{\alpha}^1\left|\left|\left| A^{\nu }XB^{1-\nu }+A^{1-\nu
}XB^{\nu }\right|\right|\right|d\nu \nonumber \\
&\leq& \frac 12\left|\left|\left| AX+XB\right|\right|\right|+\frac 12
\left|\left|\left| A^{\alpha }XB^{1-\alpha }+A^{1-\alpha
}XB^{\alpha }\right|\right|\right|\nonumber\\
&\leq& \left|\left|\left| AX+XB\right|\right|\right| \
\end{eqnarray}
where
\begin{align}
\left|\left|\left| AX+XB\right|\right|\right|&=\lim\limits_{\alpha\to
0}\frac{1}{\alpha}\int_0^{\alpha}\left|\left|\left| A^{\nu }XB^{1-\nu }+A^{1-\nu
}XB^{\nu }\right|\right|\right|d\nu\nonumber \\&=\lim\limits_{\alpha\to
1}\frac{1}{1-\alpha}\int_{\alpha}^1\left|\left|\left| A^{\nu }XB^{1-\nu
}+A^{1-\nu
}XB^{\nu }\right|\right|\right|d\nu.\nonumber \
\end{align}

\end{enumerate}
\end{proposition}

\section{Main results}

In this section, we shall prove that inequality \eqref{21} of Zhan follows immediately from the generalized version of the known inequalities \eqref{1} and \eqref{3} in the more general case of arbitrary complex Hilbert space.

\begin{theorem}\label{t1}
Let $A,B\in \mathcal{P}(\mathscr{H})$, where $\mathscr{H}$ is a Hilbert space of of arbitrary dimension and let $t\leq 2$, $r\in \lbrack \frac{1}{2},
\frac{3}{2}].$ Then for any unitarily invariant norm $\left|\left|\left|
.\right|\right|\right|$ and for every  $ X \in
\mathfrak{J}_{\left|\left|\left| .\right|\right|\right|},$ the following
inequalities hold
\begin{enumerate}
 \item for $r\in[\frac12, 1]$
\vspace{-0.2cm}\begin{align}
&\hspace{-0.5cm}\ 2\left|\left|\left|
A^{2}X+XB^{2}+tAXB\right|\right|\right|\geq
2\left|\left|\left|A^{2}X+XB^{2}+2AXB
\right|\right|\right|-(4-2t)\left|\left|\left|
AXB\right|\right|\right| \nonumber \\ &\geq
 4\left|\left|\left|
A^{\frac{3}{2}}XB^{\frac{1
}{2}}+A^{\frac{1}{2}}XB^{\frac{3}{2}}\right|\right|\right|
-(4-2t)\left|\left|\left|
AXB\right|\right|\right|\nonumber \\ &\geq
2\left|\left|\left|
A^{\frac{3}{2}}XB^{\frac{1
}{2}}+A^{\frac{1}{2}}XB^{\frac{3}{2}}\right|\right|\right|+
2\left|\left|\left| A^rXB^{2-r}+A^{
2-r}XB^{r}\right|\right|\right|-(4-2t)\left|\left|\left|
AXB\right|\right|\right|\nonumber \\
&\geq
\frac{4}{r-\frac12}\int_0^{r-\frac12}\left|\left|\left| A^{\nu
+\frac{1}{2}}XB^{\frac{3}{2}-\nu }+A^{
\frac{3}{2}-\nu }XB^{\nu +\frac{1}{2}}\right|\right|\right|d\nu
-(4-2t)\left|\left|\left|
AXB\right|\right|\right|\nonumber\\
&\geq 4\left|\left|\left| A^{\frac{2r
+1}{4}}XB^{\frac{7-2r}{4}}+A^{\frac{7-2r}{4} }XB^{\frac{2r
+1}{4}}\right|\right|\right|-(4-2t)\left|\left|\left|
AXB\right|\right|\right|\nonumber \\
&\geq 4\left|\left|\left| A^{r}XB^{2-r}+A^{
2-r}XB^{r}\right|\right|\right|-(4-2t)\left|\left|\left|
AXB\right|\right|\right|\nonumber \\
&\geq (t+2)\left|\left|\left|
A^{r}XB^{2-r}+A^{2-r}XB^{r}\right|\right|\right|\,.
\end{align}
\item for $r\in[1,\frac32]$
\vspace{-0.2cm}\begin{align}
&\hspace{-0.5cm}\ 2\left|\left|\left|
A^{2}X+XB^{2}+tAXB\right|\right|\right|\geq
2\left|\left|\left|A^{2}X+XB^{2}+2AXB
\right|\right|\right|-(4-2t)\left|\left|\left|
AXB\right|\right|\right| \nonumber \\ &\geq
 4\left|\left|\left|
A^{\frac{3}{2}}XB^{\frac{1
}{2}}+A^{\frac{1}{2}}XB^{\frac{3}{2}}\right|\right|\right|
-(4-2t)\left|\left|\left|
AXB\right|\right|\right|\nonumber \\ &\geq
2\left|\left|\left|
A^{\frac{3}{2}}XB^{\frac{1
}{2}}+A^{\frac{1}{2}}XB^{\frac{3}{2}}\right|\right|\right|+
2\left|\left|\left| A^rXB^{2-r}+A^{
2-r}XB^{r}\right|\right|\right|-(4-2t)\left|\left|\left|
AXB\right|\right|\right|\nonumber \\
&\geq
\frac{4}{\frac 32 -r}\int_{r-\frac12}^1\left|\left|\left| A^{\nu
+\frac{1}{2}}XB^{\frac{3}{2}-\nu }+A^{
\frac{3}{2}-\nu }XB^{\nu +\frac{1}{2}}\right|\right|\right|d\nu
-(4-2t)\left|\left|\left|
AXB\right|\right|\right|\nonumber\\
&\geq 4\left|\left|\left| A^{\frac{2r
+3}{4}}XB^{\frac{5-2r}{4}}+A^{\frac{5-2r}{4} }XB^{\frac{2r
+3}{4}}\right|\right|\right|-(4-2t)\left|\left|\left|
AXB\right|\right|\right|\nonumber \\
&\geq 4\left|\left|\left| A^{r}XB^{2-r}+A^{
2-r}XB^{r}\right|\right|\right|-(4-2t)\left|\left|\left|
AXB\right|\right|\right|\nonumber \\
&\geq (t+2)\left|\left|\left|
A^{r}XB^{2-r}+A^{2-r}XB^{r}\right|\right|\right|\,.
\end{align}
\end{enumerate}
\end{theorem}

\begin{proof}
Let $X \in \mathfrak{J}_{\left|\left|\left| .\right|\right|\right|}$ and
without loss of generality we may assume that $A,B\in \mathcal{P}
_{0}(\mathscr{H}). $ Put $\alpha =r-\frac{1}{2}$ then $0\leq \alpha \leq 1$.

First, we consider the case $\alpha\in [0,\frac12]$.  Using Heinz
inequality and its refinements \eqref{kit1} for unitarily invariant norms and considering $A^{-\frac 12}XB^{-\frac 12}\in
\mathfrak{J}_{\left|\left|\left| .\right|\right|\right|}$ we have
\begin{align}\label{A}
 &\hspace{-0.5cm}\left|\left|\left|
A^{\frac{1}{2}}XB^{-\frac{1
}{2}}+A^{-\frac{1}{2}}XB^{\frac{1}{2}}\right|\right|\right|  \nonumber \\
&\hspace{2cm}\geq
\frac 12\left(\left|\left|\left|
A^{\frac{1}{2}}XB^{-\frac{1
}{2}}+A^{-\frac{1}{2}}XB^{\frac{1}{2}}\right|\right|\right|+
\left|\left|\left| A^{\alpha -\frac{1}{2}}XB^{\frac{1}{2}-\alpha }+A^{
\frac{1}{2}-\alpha }XB^{\alpha
-\frac{1}{2}}\right|\right|\right|\right)\nonumber \\
&\hspace{2cm}\geq
\frac{1}{\alpha}\int_0^{\alpha}\left|\left|\left| A^{\nu -\frac{1}{2}}XB^{\frac{1}{2}-\nu }+A^{
\frac{1}{2}-\nu }XB^{\nu -\frac{1}{2}}\right|\right|\right|d\nu\nonumber\\
&\hspace{2cm}\geq\left|\left|\left| A^{\frac{\alpha
-1}{2}}XB^{\frac{1-\alpha}{2}}+A^{\frac{1-\alpha}{2} }XB^{\frac{\alpha
-1}{2}}\right|\right|\right|\nonumber \\
&\hspace{2cm}\geq \left|\left|\left| A^{\alpha
-\frac{1}{2}}XB^{\frac{1}{2}-\alpha }+A^{
\frac{1}{2}-\alpha }XB^{\alpha -\frac{1}{2}}\right|\right|\right|.
\end{align}

Since
\begin{align}
AXB^{-1}+A^{-1}XB+2X&=A^{\frac{1}{2}}(A^{\frac{
1}{2}}XB^{-\frac{1}{2}}+A^{-\frac{1}{2}}XB^{\frac{1}{2}})B^{-\frac{1}{2}}
\nonumber \\&\:+A^{-\frac{1}{2}}(A^{\frac{1}{2}}XB^{-\frac{1}{2}
}+A^{-\frac{1}{2}}XB^{\frac{1}{2}})B^{\frac{1}{2}}, \nonumber \
\end{align}
utilizing the generalized version of C-P-R inequality for unitarily invariant
norms, we
obtain
\begin{eqnarray}\label{B}
\left|\left|\left|
AXB^{-1}+A^{-1}XB+2X\right|\right|\right| \geq
2\left|\left|\left| A^{
\frac{1}{2}}XB^{-\frac{1}{2}}+A^{-\frac{1}{2}}XB^{\frac{1}{2}}\right|\right|\right|\,.
\end{eqnarray}

It follows from \eqref{A} and \eqref{B} that
\begin{align}\label{C}
&\hspace{-1.9cm}\left|\left|\left|
AXB^{-1}+A^{-1}XB+2X\right|\right|\right| \geq
 2\left|\left|\left|
A^{\frac{1}{2}}XB^{-\frac{1
}{2}}+A^{-\frac{1}{2}}XB^{\frac{1}{2}}\right|\right|\right|  \nonumber \\ &\geq
\left|\left|\left|
A^{\frac{1}{2}}XB^{-\frac{1
}{2}}+A^{-\frac{1}{2}}XB^{\frac{1}{2}}\right|\right|\right|+
\left|\left|\left| A^{\alpha -\frac{1}{2}}XB^{\frac{1}{2}-\alpha }+A^{
\frac{1}{2}-\alpha }XB^{\alpha -\frac{1}{2}}\right|\right|\right|\nonumber \\
&\geq
\frac{2}{\alpha}\int_0^{\alpha}\left|\left|\left| A^{\nu -\frac{1}{2}}XB^{\frac{1}{2}-\nu }+A^{
\frac{1}{2}-\nu }XB^{\nu -\frac{1}{2}}\right|\right|\right|d\nu\nonumber\\
&\geq 2\left|\left|\left| A^{\frac{\alpha
-1}{2}}XB^{\frac{1-\alpha}{2}}+A^{\frac{1-\alpha}{2} }XB^{\frac{\alpha
-1}{2}}\right|\right|\right|\nonumber \\
&\geq 2\left|\left|\left| A^{\alpha -\frac{1}{2}}XB^{\frac{1}{2}-\alpha }+A^{
\frac{1}{2}-\alpha }XB^{\alpha -\frac{1}{2}}\right|\right|\right|.
\end{align}
On the other hand, due to
\begin{equation*}
AXB^{-1}+A^{-1}XB+2X=AXB^{-1}+A^{-1}XB+tX+(2-t)X,
\end{equation*}
we have
\begin{align}\label{D}
 \left|\left|\left|
AXB^{-1}+A^{-1}XB+2X\right|\right|\right|\leq \left|\left|\left|
AXB^{-1}+A^{-1}XB+tX\right|\right|\right| +(2-t)\left|\left|\left|
X\right|\right|\right|\,.
\end{align}
From two last inequalities \eqref{C} and \eqref{D}, we obtain
\begin{align}\label{D'}
&\hspace{-0.1cm}2\left|\left|\left|
AXB^{-1}+A^{-1}XB+tX\right|\right|\right|\geq
 2\left|\left|\left|
AXB^{-1}+A^{-1}XB+2X\right|\right|\right|-(4-2t)\left|\left|\left|
X\right|\right|\right|\nonumber \\&\geq
 4\left|\left|\left|
A^{\frac{1}{2}}XB^{-\frac{1
}{2}}+A^{-\frac{1}{2}}XB^{\frac{1}{2}}\right|\right|\right| -(4-2t)\left|\left|\left|
X\right|\right|\right| \nonumber \\ &\geq
2\left|\left|\left|
A^{\frac{1}{2}}XB^{-\frac{1
}{2}}+A^{-\frac{1}{2}}XB^{\frac{1}{2}}\right|\right|\right|+
2\left|\left|\left| A^{\alpha -\frac{1}{2}}XB^{\frac{1}{2}-\alpha }+A^{
\frac{1}{2}-\alpha }XB^{\alpha
-\frac{1}{2}}\right|\right|\right|-(4-2t)\left|\left|\left|
X\right|\right|\right|\nonumber \\
&\geq
\frac{4}{\alpha}\int_0^{\alpha}\left|\left|\left| A^{\nu -\frac{1}{2}}XB^{\frac{1}{2}-\nu }+A^{
\frac{1}{2}-\nu }XB^{\nu -\frac{1}{2}}\right|\right|\right|d\nu -(4-2t)\left|\left|\left|
X\right|\right|\right|\nonumber\\
&\geq 4\left|\left|\left| A^{\frac{\alpha
-1}{2}}XB^{\frac{1-\alpha}{2}}+A^{\frac{1-\alpha}{2} }XB^{\frac{\alpha
-1}{2}}\right|\right|\right|-(4-2t)\left|\left|\left|
X\right|\right|\right|\nonumber \\
&\geq 4\left|\left|\left| A^{\alpha
-\frac{1}{2}}XB^{\frac{1}{2}-\alpha }+A^{
\frac{1}{2}-\alpha }XB^{\alpha -\frac{1}{2}}\right|\right|\right|-(4-2t)\left|\left|\left|
X\right|\right|\right|\,.
\end{align}
From the generalized version of C-P-R inequality for unitarily invariant norms, it is easy to
see that if $s\in \mathbb{R}$
$$4\left|\left|\left| A^sXB^{-s }+A^{-s}XB^{s
}\right|\right|\right| -4\left|\left|\left|
X\right|\right|\right| +2t\left|\left|\left|
X\right|\right|\right|
\geq (t+2)\left|\left|\left| A^sXB^{-s }+A^{-s}XB^{s
}\right|\right|\right|\,.
$$
From \eqref{D'} and the last  inequality, we can deduce that for any $X\in \mathfrak{J}_{\left|\left|\left| .\right|\right|\right|}$
\begin{align}
& 2\left|\left|\left|
AXB^{-1}+A^{-1}XB+tX\right|\right|\right|
\geq 2\left|\left|\left|AXB^{-1}  +A^{-1}XB+2X \right|\right|\right|-(4-2t)\left|\left|\left|
X\right|\right|\right| \nonumber \\ &\geq
 4\left|\left|\left|
A^{\frac{1}{2}}XB^{-\frac{1
}{2}}+A^{-\frac{1}{2}}XB^{\frac{1}{2}}\right|\right|\right| -(4-2t)\left|\left|\left|
X\right|\right|\right| \nonumber \\ &\geq
2\left|\left|\left|
A^{\frac{1}{2}}XB^{-\frac{1
}{2}}+A^{-\frac{1}{2}}XB^{\frac{1}{2}}\right|\right|\right|+
2\left|\left|\left| A^{\alpha -\frac{1}{2}}XB^{\frac{1}{2}-\alpha }+A^{
\frac{1}{2}-\alpha }XB^{\alpha
-\frac{1}{2}}\right|\right|\right|-(4-2t)\left|\left|\left|
X\right|\right|\right|\nonumber \\
&\geq
\frac{4}{\alpha}\int_0^{\alpha}\left|\left|\left| A^{\nu -\frac{1}{2}}XB^{\frac{1}{2}-\nu }+A^{
\frac{1}{2}-\nu }XB^{\nu -\frac{1}{2}}\right|\right|\right|d\nu -(4-2t)\left|\left|\left|
X\right|\right|\right|\nonumber\\
&\geq 4\left|\left|\left| A^{\frac{\alpha -1}{2}}XB^{\frac{1-\alpha}{2}}+A^{\frac{1-\alpha}{2} }XB^{\frac{\alpha -1}{2}}\right|\right|\right|-(4-2t)\left|\left|\left|
X\right|\right|\right|\nonumber \\
&\geq 4\left|\left|\left| A^{\alpha -\frac{1}{2}}XB^{\frac{1}{2}-\alpha }+A^{
\frac{1}{2}-\alpha }XB^{\alpha -\frac{1}{2}}\right|\right|\right|-(4-2t)\left|\left|\left|
X\right|\right|\right|\nonumber \\&\geq
(t+2)\left|\left|\left|
A^{\alpha -\frac{1}{2}}XB^{\frac{1}{2}-\alpha }+A^{\frac{1}{2}-\alpha
}XB^{\alpha -\frac{1}{2}}\right|\right|\right|\,.\
\end{align}
whence, by replace $X$ by $AXB$ and $\alpha$ by $r-\frac12$, we get
\begin{align}
&\hspace{-0.5cm}\ 2\left|\left|\left|
A^{2}X+XB^{2}+tAXB\right|\right|\right|\geq
2\left|\left|\left|A^{2}X+XB^{2}+2AXB
\right|\right|\right|-(4-2t)\left|\left|\left|
AXB\right|\right|\right| \nonumber \\ &\geq
 4\left|\left|\left|
A^{\frac{3}{2}}XB^{\frac{1
}{2}}+A^{\frac{1}{2}}XB^{\frac{3}{2}}\right|\right|\right|
-(4-2t)\left|\left|\left|
AXB\right|\right|\right|\nonumber \\ &\geq
2\left|\left|\left|
A^{\frac{3}{2}}XB^{\frac{1
}{2}}+A^{\frac{1}{2}}XB^{\frac{3}{2}}\right|\right|\right|+
2\left|\left|\left| A^rXB^{2-r}+A^{
2-r}XB^{r}\right|\right|\right|-(4-2t)\left|\left|\left|
AXB\right|\right|\right|\nonumber \\
&\geq
\frac{4}{r-\frac12}\int_0^{r-\frac12}\left|\left|\left| A^{\nu
+\frac{1}{2}}XB^{\frac{3}{2}-\nu }+A^{
\frac{3}{2}-\nu }XB^{\nu +\frac{1}{2}}\right|\right|\right|d\nu
-(4-2t)\left|\left|\left|
AXB\right|\right|\right|\nonumber\\
&\geq 4\left|\left|\left| A^{\frac{2r
+1}{4}}XB^{\frac{7-2r}{4}}+A^{\frac{7-2r}{4} }XB^{\frac{2r
+1}{4}}\right|\right|\right|-(4-2t)\left|\left|\left|
AXB\right|\right|\right|\nonumber \\
&\geq 4\left|\left|\left| A^{r}XB^{2-r}+A^{
2-r}XB^{r}\right|\right|\right|-(4-2t)\left|\left|\left|
AXB\right|\right|\right|\nonumber \\
&\geq (t+2)\left|\left|\left|
A^{r}XB^{2-r}+A^{2-r}XB^{r}\right|\right|\right|\,.
\end{align}

Finally, we note that the case $\alpha\in [\frac12, 1]$ is obtained analogously
and this completes the proof.
\end{proof}

Note that the case $t\leq -2$ is trivial. An immediate consequence for the
case $r=1$ of this last theorem is the following (exactly the corollary 7 in
\cite{12} in finite dimensional case).

\begin{corollary}
Let $A,B\in \mathbb{B}(\mathscr{H})$ and let $t\leq 2.$ Then
\begin{eqnarray}\label{23}
\forall X\in \mathfrak{J}_{\left|\left|\left| .\right|\right|\right|},\ \left|\left|\left| A^*AX+XBB^{*
}+t\left\vert A\right\vert X\left\vert B\right\vert \right|\right|\right|
\geq (t+2)\left|\left|\left| AXB^*\right|\right|\right|\,.
\end{eqnarray}
\end{corollary}

Another immediate consequence of this last corollary is (exactly the corollary
8 in \cite{12} in finite dimensional case).

\begin{corollary}
Let $P,Q\in \mathcal{P}_{0}(\mathscr{H})$ and let $t\leq 2.$ Then
\begin{eqnarray}\label{24}
\forall X\in \mathfrak{J}_{\left|\left|\left| .\right|\right|\right|},\ \left|\left|\left|
PXQ^{-1}+P^{-1}XQ+tX\right|\right|\right| \geq (t+2)\left|\left|\left|
X\right|\right|\right|
\end{eqnarray}
\end{corollary}

\begin{remark}
The last theorem and their two consequences was proved by Zhan in \cite{12} in
the particular case of finite dimensional case. Note that
Cano--Mosconi--Stojanoff \cite{1} have proved the last corollary using the
spectral measure of a normal operator to generalize \cite[Corollary 8]{12} of Zhan for
arbitrary complex Hilbert space. Here, we have proved it in a general situation for an arbitrary Hilbert space using only
known operator inequalities.
\end{remark}

\begin{remark}
It follows from the above corollary that for every $k\leq 2$ and for every
operator $S\in \mathbb{C}^*\mathcal{P}_{0}(\mathscr{H})$ the following
inequality holds
\begin{eqnarray}\label{25}
\forall X\in \mathbb{B}(\mathscr{H}),\ \left\Vert SXS^{-1}+S^{-1}XS+kX\right\Vert
\geq (k+2)\left\Vert X\right\Vert
\end{eqnarray}

So it is interesting to characterize the class of all operators $S$ in $
\mathfrak{I}(\mathscr{H})$ satisfying this last inequality. We denote this class by $
\mathfrak{D}_{k}(\mathscr{H})$.
\end{remark}

\begin{proposition}\label{above}
For every real numbers $k,\ t$,

(i) if $k\geq t$, then $\mathfrak{D}_{k}(\mathscr{H})\subset \mathfrak{D}_{t}(\mathscr{H})$,

(ii) if $k\geq 0$, then $$\mathfrak{D}_{k}(\mathscr{H})\subset \left\{ \alpha
S:\alpha \in \mathbb{C}^*,\ S\in \mathscr{S}_{0}(\mathscr{H})\text{, }\left\vert
\frac{\lambda }{\mu }+\frac{\mu }{\lambda }+k\right\vert \geq k+2,\ \lambda
,\mu \in \sigma (S)\right\} .$$
\end{proposition}

\begin{proof}
(i) This follows by the same argument as in the proof of Theorem \ref{t1}.

(ii) Let $S\in \mathfrak{D}_{k}(\mathscr{H}).$ It follows immediately that the
inequality $\left\Vert SXS^{-1}+S^{-1}XS\right\Vert \geq 2\left\Vert
X\right\Vert $ holds for every $X$ in $\mathbb{B}(\mathscr{H}).$ Thus $S\in \mathbb{C
}^*\mathscr{S}_{0}(\mathscr{H}).$

We may assume without loss of generality that $S$ is invertible and
self-adjoint. Denote by $\varphi _{S,k}$ the operator on $\mathbb{B}(\mathscr{H})$
given by $\varphi _{S,k}(X)=SXS^{-1}+S^{-1}XS+kX.$ So that $\sigma (\varphi
_{S,k})=\left\{ \frac{\lambda }{\mu }+\frac{\mu }{\lambda }+k:\lambda ,\mu
\in \sigma (S)\right\} \subset \mathbb{R}.$ Hence each spectral value of $
\varphi _{S,k}$ is in an approximate point value. Let $\lambda ,\mu \in
\sigma (S).$ Then there exists a sequence $(X_{n})$ of operators of norm
one such that $\left\Vert SX_{n}S^{-1}+S^{-1}X_{n}S+kX_{n}\right\Vert
\rightarrow \left\vert \frac{\lambda }{\mu }+\frac{\mu }{\lambda }
+k\right\vert .$ Thus $k+2=\inf_{\left\Vert X\right\Vert =1}\left\Vert
SXS^{-1}+S^{-1}XS+kX\right\Vert \leq \left\vert \frac{\lambda }{\mu }+\frac{
\mu }{\lambda }+k\right\vert .$
\end{proof}

\begin{remark}
In the case where $k\geq 0$ and $\dim \mathscr{H}=2$, the inclusion given in the above
proposition becomes an equality$.$

Indeed, let $S$ be an invertible self-adjoint operator in $\mathbb{B}(\mathscr{H}),$
and let $\lambda $ and $\mu $ be the eigenvalues of $S$ such that $\left\vert
\frac{\lambda }{\mu }+\frac{\mu }{\lambda }+k\right\vert \geq k+2.$ By a
simple computation, we obtain
\begin{equation*}
\forall X\in \mathbb{B}(\mathscr{H}),\ SXS^{-1}+S^{-1}XS+kX=\left(
\begin{array}{cc}
k+2 & \frac{\lambda }{\mu }+\frac{\mu }{\lambda }+k \\
\frac{\lambda }{\mu }+\frac{\mu }{\lambda }+k & k+2
\end{array}
\right) \circ X
\end{equation*}

Since the matrix $\left(
\begin{array}{cc}
\frac{1}{k+2} & 1\left/ \left( \frac{\lambda }{\mu }+\frac{\mu }{\lambda }
+k\right) \right. \\
1\left/ \left( \frac{\lambda }{\mu }+\frac{\mu }{\lambda }+k\right) \right.
& \frac{1}{k+2}
\end{array}
\right) $ is positive definite, thus using the Schur theorem, we obtain
\begin{equation*}
\forall X\in \mathbb{B}(\mathscr{H}),\left\Vert \left(
\begin{array}{cc}
\frac{1}{k+2} & 1\left/ \left( \frac{\lambda }{\mu }+\frac{\mu }{\lambda }
+k\right) \right. \\
1\left/ \left( \frac{\lambda }{\mu }+\frac{\mu }{\lambda }+k\right) \right.
& \frac{1}{k+2}
\end{array}
\right) \circ X\right\Vert \leq \frac{1}{k+2}\left\Vert X\right\Vert
\end{equation*}

Therefore
\begin{equation*}
\forall X\in \mathbb{B}(\mathscr{H}),\left\Vert \left(
\begin{array}{cc}
k+2 & \frac{\lambda }{\mu }+\frac{\mu }{\lambda }+k \\
\frac{\lambda }{\mu }+\frac{\mu }{\lambda }+k & k+2
\end{array}
\right) \circ X\right\Vert \geq (k+2)\left\Vert X\right\Vert
\end{equation*}
\end{remark}

\begin{conjecture}\label{conj}
Let $k$ be a real number such that $0\leq k\leq 2.$ Then for every natural
number $n$ and for every nonzero numbers $\lambda _{1},\dots,\lambda _{n}$
such that $\left\vert \frac{\lambda _{i}}{\lambda _{j}}+\frac{\lambda _{j}}{
\lambda _{i}}+k\right\vert \geq k+2$, for $i,j=1, \dots, n$ the matrix $\left(
\frac{\lambda _{i}\lambda _{j}}{\lambda _{i}^{2}+\lambda _{j}^{2}+k\lambda
_{i}\lambda _{j}}\right) _{1\leq i,j\leq n}$ is positive.
\end{conjecture}

\noindent Furthermore,

\begin{theorem}
Assume that Conjecture \ref{conj} is valid and $\dim \mathscr{H}=n$. Then for every number $k$ such that $0\leq k\leq 2$,
$$\mathfrak{D}_{k}(\mathscr{H})=\left\{ \alpha S:\alpha \in \mathbb{C}^*,\
S\in \mathscr{S}_{0}(\mathscr{H})\text{, }\left\vert \frac{\lambda }{\mu }+\frac{\mu }{
\lambda }+k\right\vert \geq k+2,\ \lambda ,\mu \in \sigma (S)\right\} .$$
\end{theorem}

\begin{proof}
Using Proposition \ref{above}, it remains to prove that $$\left\{ \alpha
S:\alpha \in \mathbb{C}^*,\ S\in \mathscr{S}_{0}(\mathscr{H})\text{, }\left\vert
\frac{\lambda }{\mu }+\frac{\mu }{\lambda }+k\right\vert \geq k+2,\ \lambda
,\mu \in \sigma (S)\right\} \subset \mathfrak{D}_{k}(\mathscr{H}).$$ This follows by using
the same argument used in the Remark.
\end{proof}

Finally we present new variants of C-P-R inequality.

\begin{theorem} \label{t2}
Let $S \in \mathfrak{I}(\mathscr{H})$ and $X, Y \in
\mathfrak{J}_{\left|\left|\left| .\right|\right|\right|}.$
The following inequality holds and is equivalent to the C-P-R inequality for unitarily invariant norms:
\begin{eqnarray}\label{mos1}
{\rm (i)} \left|\left|\left| \left(SYS^{-1}+S^{* -1}YS^*\right) \oplus \left(S^*XS^{* -1}+S^{-1}XS\right)\right|\right|\right| \geq 2 \left|\left|\left| X\oplus Y\right|\right|\right|\,;
\end{eqnarray}
\begin{eqnarray}\label{mos2}
{\rm (ii)} \left|\left|\left| \left(SYS^{* -1}+S^{* -1}YS\right) \oplus \left(S^*XS^{-1}+S^{-1}XS^*\right)\right|\right|\right| \geq 2 \left|\left|\left| X\oplus Y\right|\right|\right|\,.
\end{eqnarray}
\end{theorem}
\begin{proof}
(i) Clearly $\left[\begin{array}{cc} 0 & S \\ S^* & 0 \end{array}\right]$ is a self adjoint operator in $\mathbb{B}(\mathscr{H}\oplus\mathscr{H})$ and $\left[\begin{array}{cc} 0 & S \\ S^* & 0 \end{array}\right]^{-1}=\left[\begin{array}{cc} 0 & S^{* -1} \\ S^{-1} & 0 \end{array}\right]$. It follows from the C-P-R inequality for unitarily invariant norms that
\begin{align*}
\left|\left|\left|\left[\begin{array}{cc} 0 & S \\ S^* & 0 \end{array}\right] \left[\begin{array}{cc} X & 0 \\ 0 & Y \end{array}\right] \left[\begin{array}{cc} 0 & S \\ S^* & 0 \end{array}\right]^{-1} + \left[\begin{array}{cc} 0 & S \\ S^* & 0 \end{array}\right]^{-1}\left[\begin{array}{cc} X & 0 \\ 0 & Y  \end{array}\right]\left[\begin{array}{cc} 0 & S \\ S^* & 0 \end{array}\right] \right|\right|\right|\\ \geq 2 \left|\left|\left|\left[\begin{array}{cc} X & 0 \\ 0 & Y  \end{array}\right] \right|\right|\right|\,,
\end{align*}
whence
\begin{align*}
\left|\left|\left|\left[\begin{array}{cc} SYS^{-1}+S^{* -1}YS^* & 0 \\ 0 & S^*XS^{* -1}+S^{-1}XS \end{array}\right]\right|\right|\right| \geq 2 \left|\left|\left|\left[\begin{array}{cc} X & 0 \\ 0 & Y  \end{array}\right] \right|\right|\right|\,,
\end{align*}
which is indeed \eqref{mos1}. Now assume that \eqref{mos1} holds for all $X, Y \in \mathfrak{J}_{\left|\left|\left| .\right|\right|\right|}$ and $S \in \mathfrak{I}(\mathscr{H})$. To get \eqref{2} for unitarily invariant norms, let $S$ be self-adjoint, take $Y=X$ and use the fact that two inequalities $|||A|||\leq |||B|||$ and $|||A\oplus A|||\leq |||B\oplus B|||$, by the Fan dominance principle, are equivalence for all unitarily invariant norms (see \cite{Ki}).\\
(ii) To get inequality \eqref{mos2}, use the same argument as in (i) with the matrix $\left[\begin{array}{cc} 0 & X \\ Y & 0  \end{array}\right]$ and note that $|||X\oplus Y|||=\left|\left|\left|\left[\begin{array}{cc} 0 & X \\ Y & 0 \end{array}\right]\right|\right|\right|$.
\end{proof}

\begin{corollary}
(i) If $S \in \mathfrak{I}(\mathscr{H})$ and $X \in \mathbb{B}(\mathscr{H})$, then the following inequality holds and is equivalent to the C-P-R inequality
\begin{eqnarray*}
\max\{\Vert SXS^{-1}+S^{* -1}XS^* \Vert, \Vert S^*XS^{* -1}+S^{-1}XS \Vert\} \geq 2 \Vert X\Vert\,.
\end{eqnarray*}
(ii) If $S \in \mathfrak{I}(\mathscr{H})$ and $X, Y$ are in the Schatten $p$-class, then the following inequality holds and is equivalent to the C-P-R inequality for the Schatten $p$-norm
\begin{eqnarray*}
\left|\left|\left| SXS^{-1}+S^{* -1}XS^*\right|\right|\right|_p^p + \left|\left|\left| S^*XS^{* -1}+S^{-1}XS\right|\right|\right|_p^p \geq 2^{p+1} \left|\left|\left| X\right|\right|\right|_p^p \,.
\end{eqnarray*}
\end{corollary}
\begin{proof}
Apply \eqref{mos1} to $Y=X$ and equalities \eqref{erf}.
\end{proof}

\end{document}